\documentclass[12pt,4apaper]{amsart}
\usepackage[cp1251]{inputenc}
\usepackage[english]{babel}

\usepackage{amsmath,amsfonts,amssymb,mathrsfs,amscd,comment,latexsym,bbold}
\usepackage[matrix,arrow,curve]{xy}

\usepackage[usenames]{color}
\usepackage{colortbl}
\usepackage[dvipsnames]{xcolor}

\usepackage{enumitem}

\usepackage[top = 3cm, bottom = 4cm, left = 3cm, right = 3cm]{geometry}

\usepackage{amsthm}

\theoremstyle{plain}
\newtheorem{theorem}{Theorem}
\newtheorem*{nonum-theorem}{Theorem}
\newtheorem{proposition}{Proposition}
\newtheorem{lemma}{Lemma}
\newtheorem{lemma-remark}{Lemma-remark}
\newtheorem{sublemma}{Sublemma}[lemma]
\newtheorem{corollary}{Corollary}
\newtheorem{nonum-corollary}{Corollary}

\newtheoremstyle{handleNumber}{}{}{\itshape}{}{}{}{\newline}{{\bf #1} \thmnote{#3}}
\theoremstyle{handleNumber}
\newtheorem*{handnum-theorem}{Theorem}

\theoremstyle{definition}
\newtheorem{definition}{Definition}

\newtheorem*{nonum-definition}{Definition}

\theoremstyle{remark}
\newtheorem{remark}{Remark}

\renewcommand{\leq}{\leqslant}
\renewcommand{\geq}{\geqslant}

\begin{document}

\newcommand{\Spec}{\mathrm{Spec}\,}
\newcommand{\codim}{\mathrm{codim}\,}
\newcommand{\hocoeq}{\mathrm{hocoeq}\,}
\newcommand{\colim}{\operatorname{colim}}

\newcommand{\Fr}{\mathrm{Fr}}
\newcommand{\ZF}{\mathbb Z\mathrm{F}}
\newcommand{\qf}{\mathrm{qf}}
\newcommand{\Frqf}{\Fr^\qf}

\newcommand{\A}{\mathbb A}

\newcommand{\fr}{\mathrm{fr}}
\newcommand{\nis}{\mathrm{nis}}

\newcommand{\Sm}{\mathrm{Sm}}
\newcommand{\AffSm}{\mathrm{AffSm}}
\newcommand{\Shv}{\mathrm{Shv}}
\newcommand{\sShv}{\mathrm{sShv}}

\newcommand{\PP}{\mathbb P}
\newcommand{\Gm}{\mathbb G_m}
\newcommand{\sMSfr}{\mathcal{M}^{S^1}_\fr}
\newcommand{\MSfr}{\mathrm{M}^{S^1}_\fr}
\newcommand{\sMGcSfr}{\mathcal{M}^{(\Gm,S^1)}_\fr}
\newcommand{\MGcSfr}{\mathrm{M}^{(\Gm,S^1)}_\fr}
\newcommand{\sMGwSfr}{\mathcal{M}^{(\Gm^{\wedge 1}\wedge S^1)}_\fr}
\newcommand{\MGwSfr}{\mathrm{M}^{(\Gm^{\wedge 1}\wedge S^1)}_\fr}
\newcommand{\sMcTfr}{\mathcal{M}^{(\A^1//\Gm)}_\fr}
\newcommand{\McTfr}{\mathrm{M}^{(\A^1//\Gm)}_\fr}
\newcommand{\sMTfr}{\mathcal{M}^{T}_\fr}
\newcommand{\MTfr}{\mathrm{M}^{T}_\fr}
\newcommand{\sMPfr}{\mathcal{M}^{\PP^1}_\fr}
\newcommand{\MPfr}{\mathrm{M}^{\PP^1}_\fr}
\newcommand{\Mfr}{\mathrm{M}_{\fr}}
\newcommand{\MP}{\mathrm{M}_{\PP^1}}
\newcommand{\sMGfr}{\mathcal{M}^{\Gm}_\fr}
\newcommand{\MGfr}{\mathrm{M}^{\Gm}_\fr}

\newcommand{\cVm}{\stackrel{m}{\mathcal V}}
\newcommand{\cEm}{\stackrel{m}{\mathcal E}}
\newcommand{\cSm}{\stackrel{m}{S}}
\newcommand{\pr}{\mathrm{pr}}



\title{Framed motives of smooth affine pairs.}


\author{A.~Druzhinin}

\begin{abstract}
The theory of framed motives by Garkusha and Panin gives
computations in the stable motivic homotopy category $\mathbf{SH}(k)$
in terms of Voevodsky's framed correspondences.
In particular 
the  motivically fibrant $\Omega$-resolution in positive degrees 
of the 
motivic suspension spectrum $\Sigma_{\mathbb P^1}^\infty X_+$, where $X_+=X\amalg *$,
for a smooth scheme $X\in \Sm_k$ over an infinite perfect field $k$,
is computed. 

The 
computation by Garkusha, Neshitov and Panin
of the framed motives of relative motivic spheres 
$(\A^l\times X,(\A^l-0)\times X)$, $X\in \Sm_k$,
is one of ingredients in the theory. 
In the article we extend this result 
to the case of 
a pair $(X,U)$ given by a smooth affine variety $X$ over $k$
and an open subscheme $U\subset X$.

The result gives the explicit  motivically fibrant $\Omega$-resolution in positive degrees 
for the  motivic suspension spectrum $\Sigma_{\mathbb P^1}^\infty (X_+/U_+)$ of the factor-sheaf $X_+/U_+$. 

%
\end{abstract}

%

\address{
Chebyshev Laboratory, St. Petersburg State University, 
14th Line V.O., 29B, 
Saint Petersburg 199178 Russia; 
andrei.druzh@gmail.com
}



\maketitle

\section{Introduction}
\label{sect:Intr}

In the unpublished notes \cite{VoevFrCor}
Voevodsky 
had suggested the computational approach to the Morel-Voevodsky 
stable motivic homotopy category $\mathbf{SH}(k)$ 
\cite{Morel-Voevodsky}, \cite{Jardine-spt}, \cite{Mor0}
over a perfect base field.
Realising this 
idea Garkusha and Panin 
constructed the theory of framed motives over an infinite perfect field $k$, 
see \cite{GP_MFrAlgVar} and \cite{AGP-FrCanc},\cite{GP-HIVth}, \cite{GNP_FrMotiveRelSphere}. 
The 
aim of the present article 
is to extend 
the computations
\cite[th. 4.1, 11.1]{GP_MFrAlgVar} 
of
stable motivic fibrant resolutions of 
the motives 
of 
smooth schemes $X$ 
to the case of 
open pair factor-sheaf 
$X/U$ 
with $X$ smooth affine. 

\newcommand{\arSmp}{\Delta^\mathrm{op}\overrightarrow{\Sm_+}}


In \cite{VoevFrCor} for any smooth scheme 
$X$ 
and open $U\subset X$
Voevodsky had introduced 
pointed sheaves $\Fr(-, X/U)\in \mathrm{Shv}_\bullet$
of 
so-called 
stable
Voevodsky's 
framed correspondences.
According to the definition \cite[def. 2.8]{GP_MFrAlgVar} and fundamental Voevodsky's lemma \cite[lm. 3.5, prop. 3.2]{GP_MFrAlgVar} 
\begin{multline*}
\Fr( - , X/U ) = \varinjlim_n\Fr_n( - , X/U ),\\
\Fr_n( - , X/U )=
Hom_{\mathrm{Shv}_\bullet}(-\wedge (\PP^1/\infty)^{\wedge n}, X/U\wedge (\A^1/\Gm)^{\wedge n}),
\end{multline*}
where $\mathrm{Shv}_\bullet$ denotes the pointed Nisnevich sheaves,
and 
sections of $\Fr_n(-,X/U)$ 
have precise geometrical description, see def. \ref{def:FrCor} or \cite[def. 2.5]{GP_MFrAlgVar}.

In \cite[theorem 11.7]{GP_MFrAlgVar} 
Garkusha and Panin
computed the functor 
\begin{equation}\label{eq:intr:OmGSigGSigS}\Omega^\infty_{\Gm}\Sigma^\infty_{\Gm}\Sigma^\infty_{S^1}\colon \mathbf{H}_\bullet(k)\to \mathbf{SH}_{S^1}(k).\end{equation}
from the unstable pointed motivic homotopy category to 
the $S^1$-stable one
as the functor 
\[\Mfr\colon \sShv_\bullet\to \Spec_{S^1}( \sShv_\bullet )\]
from the category of pointed simplicial sheaves to the motivic $S^1$-spectra,
such that by \cite[ths. 10.1(2), 7.6]{GP_MFrAlgVar} 
the level injective local fibrant replacement $\Mfr(-)_f$ 
lands in motivically fibrant $\Omega_{S^1}$-spectra in positive degrees.

For a smooth scheme $X$ over $k$ 
the spectrum $\Mfr(X)$ is equal to
\begin{equation}\label{eq:intr:SpectrumFormula(Mfr(X))} 
(C^*\Fr(-, X ), C^*\Fr(-,X\wedge S^1), \dots C^*\Fr(-,X\wedge S^i)\dots),
\end{equation}
where
the endo-functor
$C^*\colon F (-)
\mapsto F(\Delta^\bullet\times -)$ 
on 
$\sShv_\bullet$
is corepresented by the co-simplicial scheme $\Delta^\bullet$, 
$\Delta^n=\{(x_0,\dots x_n)\in \A^{n+1}, x_1+\dots x_n=1\}$, 
and the pointed simplicial sheaf $\Fr(-,\mathcal Y)$, 
for a simplicial scheme $\mathcal Y$,
is defined by the pointed sheaves $\Fr(-,Y)$,
for $Y\in \Sm_k$.

In the article we 
extend so-called Cone Theorem by Garkusha, Neshitov and Panin \cite{GNP_FrMotiveRelSphere},
and
prove that 
for a smooth affine $X$ over $k$
and open $U\subset X$,
the spectrum $\Mfr(X_+/U_+)_f$, where $X_+/U_+$ is the pointed factor-sheaf,
is schemewise simplicially weak equivalent to
\begin{equation}\label{eq:intr:SpectrumFormula(Mfr(X,U)f)} 
(C^*\Fr(-, X/U )_f, C^*\Fr(-,X/U\wedge S^1)_f, \dots C^*\Fr(-,X/U\wedge S^i)_f\dots),
\end{equation}
where $(-)_f$ denotes the fibrant replacement within the (level) injective local model structure. 
The result of \cite{GNP_FrMotiveRelSphere} 
covers the case of 
sheaves 
$( \A^n/ (\A^n-0)) \times X$, $X\in \Sm_k$.

As a consequence 
this proves the results of \cite[th. 4.1, 11.1, 11.7]{GP_MFrAlgVar} for the case of affine pairs.
\begin{theorem}[corollares \ref{cor:MPmotresolution}, \ref{cor:MGfrresolution}, \ref{cor:MfrcomputinfloopSH^S1}]\label{th:intr:MPMfrmotresolutions}
For a smooth affine scheme $X$ over an infinite perfect field $k$, and an open subscheme $U\subset X$,
the following holds:

1) The canonical morphism of $\PP^1$-spectra of pointed simplicial sheaves
\[\Sigma^\infty_{\PP^1}(X/U)\to \MP(X,U),\]
where the spectrum $\MP(X,U)$ is given by
\begin{equation}\label{eq:intr:MP}(C^*\Fr(-,X/U), C^*\Fr(-,X/U\wedge T), \dots C^*\Fr(-,X/U\wedge T^{\wedge i})),\end{equation}
where $T=(\A^1,\Gm)$,
is a stable motivic weak equivalence in $\mathbf{SH}(k)$,
and the Nisnevich local fibrant replacement $\MP(X,U)_f$ is a motivically fibrant $\Omega_{\PP^1}$-spectrum in positive degrees.

2) 
The $S^1$-spectrum of pointed simplicial sheaves $\Mfr(X,U)_f$ given by \eqref{eq:intr:SpectrumFormula(Mfr(X,U)f)}
is a motivically fibrant $\Omega$-spectrum in positive degrees
and has the homotopy type of the spectrum $\Omega^\infty_{\mathbb G_m}\Sigma^\infty_{\mathbb G_m}\Sigma^\infty_{S^1}(X/U)$ in $\mathbf{SH}_{S^1}(k)$.

3)
Let $\MGfr(X,U)_f$ be a $(S^1,\Gm^{\wedge 1})$-bi-spectrum given by
\begin{equation}\label{eq:intr:MGfr(X,U)f}(\Mfr(X,U)_f,\Mfr((X,U)\wedge \Gm^{\wedge 1})_f ,\dots \Mfr((X,U)\wedge \Gm^{\wedge i} )_f\dots ).\end{equation}
Then the canonical morphism of bi-spectra \[\Sigma^\infty_{\Gm}\Sigma^\infty_{S^1}(X/U)\to \MGfr(X,U)_f\] is a stable motivic weak equivalence, and $\MGfr(X,U)_f$ is a motivically fibrant $\Omega$-bi-spectrum in $S^1$-positive degrees.
\end{theorem}

By the definition 
it follows that $\Mfr(X/U)=\Mfr(X//U)$, where $X//U$
is the simplicial cone of the open immersion $j\colon U\hookrightarrow X$,
that is the colimit 
in the category of pointed simplicial schemes
\begin{equation}\label{eq:intr:SimplCone}
X//U=\colim(X \xleftarrow{j} U \xrightarrow{i_1} \Delta^1_s\times U\xleftarrow{i_0} U\rightarrow * )\in \Delta^\mathrm{op}\Sm_\bullet
,
\text{ if } U\neq\emptyset,
\end{equation} 
where
$\Delta^1_s$ is the simplicial interval,
$i_0,i_1\colon U\to \Delta^1_s\times U$ are the unit and zero faces,
and $*$ is a rational point, that is the base point of $X//U$.
So the mentioned Cone Theorem is equivalent to the level Nisnevich local equivalence 
\begin{equation}\label{eq:intr:ConeTh}
\Mfr(X//U)\simeq_{\mathrm{nis}} \Mfr(X,U) \in \Spec_{S^1}(\sShv_\bullet),
\end{equation}
where $\Mfr(X,U)$ is the motivic spectrum with terms $C^*\Fr(-,(X/U)\wedge S^i)$.

In distinct to 
\cite{GNP_FrMotiveRelSphere} our arguments for 
\eqref{eq:intr:ConeTh} 
are unstable, 
so we prove the 
equivalence \eqref{eq:intr:ConeTh} 
levelwise independently. 
\begin{theorem}[Cone Theorem, Theorem \ref{th:ConeTh}]\label{th:intr:ConeTh}
For a smooth affine pair $(X,U)$ over an infinite field $k$ 
let $X//U$ denote 
the simplicial cone of the morphism $U\to X$, see \eqref{eq:intr:SimplCone} for the case $U\neq \emptyset$.

Then the canonical morphism of simplicial pointed sheaves \[\epsilon\colon \Fr(-,X//U)\to \Fr(-,X/U),\]
is a motivic equivalence.
\end{theorem}
\begin{remark}
Note that 
the pointed simplicial sheaf $\Fr(-,X//U)$ is simplicially equivalent to $\Fr(-, X )/\Fr(-, U )$, 
and
the morphism \[\Fr(-, X )/\Fr(-, U )\to \Fr(-, X/U )\]
is not a simplicial equivalence, but is a motivic one by the above. 
\end{remark}
The result is given by the composition of motivic equivalences 
\begin{equation}\label{eq:intr:FrX//YFrqfFr}\Fr(-,X//U) \to \Frqf(-,X/U) \to \Fr(-,X/U),\end{equation}
where the first morphism is a Nisnevich local equivalence, the second one is an $\A^1$-homotopy-equivalence. 
The mid-term $\Frqf$ is the subsheaf of $\Fr$
given by
so-called quasi-finite framed correspondences, introduced in \cite{GNP_FrMotiveRelSphere}, see definition \ref{sect:QuasiFiniteCor} in section \ref{def:Frqf} in our text.

The theory of framed motives was revisited by the team of five authors EHKSY in \cite{ElHoKhSoYa-MotDeloop} with the technique of $\infty$-categories;
 this allowed to obtain clear $\infty$-categorical picture of the theory and obtain the universal properties of the constructed categories. 

Here we follow the original approach \cite{GP_MFrAlgVar},
since 
it is more explicit geometrical, and closure to the precise computations.
The main content and novelty of the present work relates to the geometrical arguments in the proof of the moving lemma, section  \ref{sect:MovingLemma}, providing the second equivalence  of \eqref{eq:intr:FrX//YFrqfFr}.

In the same time let us point that the definition of the tangentially framed correspondences introduced in \cite{ElHoKhSoYa-MotDeloop} and its extend to the case of pairs $(X,U)$
is useful for the proof of proposition \ref{prop:FrX/FrUeqnisFrPair} and corollary \ref{cor:FrPairFreqNisEq},
because the correspondences of this type are $\infty$-commutative monoids already on the simplicial level (with out motivic localisation).
Let us point that 
originally in \cite{GNP_FrMotiveRelSphere}
the first equivalence in \eqref{eq:intr:FrX//YFrqfFr} was proven for linear sheaves $\ZF$ and $\ZF^\qf$ instead of $\Fr$ and $\Frqf$, because the linear ones are commutative monoids, but $\Fr$ are $\Frqf$ are not.


\subsection{Remarks on the generality of the base filed.} 



Originally the theory of framed motives \cite{GP_MFrAlgVar} was written with the exceptional case of $\mathrm{char}\, k=2$, 
but was extended to this case by \cite{DP-Surjetex}. 
That is way we 
use the results of the theory \cite{GP_MFrAlgVar} in all characteristics.

The theory of famed motives was extended to the case of finite fields in \cite{DK18} and \cite{ElHoKhSoYa-MotDeloop}.
The same arguments extend our result for the case of finite fields as well, but we don't touch this here.


\subsection{Acknowledgement}
The research is supported by the Russian Science Foundation grant 19-71-30002.
Acknowledgement to I.~Panin for that he encouraged me to work on this project and for helpful discussions.

\subsection{Notation.}
$\mathrm{Sm}_k$ denotes the category of smooth separated schemes of finite type over the base field $k$. 
$\mathrm{AffSm}_k\subset \Sm_k$ denotes the subcategory of affine smooth schemes.

For a scheme $X$ and a vector of regular functions $\varphi=(\varphi_i)$, $\varphi\in \mathcal O(X)$, we denote by $Z(\varphi)$ the common vanishing locus of functions $\varphi_i$.

For an $S$-scheme $X$ we write $\dim_S X=r$ iff the Krull dimension of $X\times_S \sigma$ is equal to $r$ for any point $\sigma\in S$.
We write $\dim_S X\geq r$ ($\dim_S X\leq r$) iff $\dim (X\times_S \sigma)\geq r$ ($\dim (X\times_S \sigma)\leq r$) for any $\sigma\in S$.
For a morphism of $S$-schemes $X\to Y$ we denote $\dim_S (X/Y) = \dim_S X - \dim_S Y$, and $\codim_S (X/Y) = \dim_S Y - \dim_S X$ if $X$ and $Y$ are of constant dimension over $S$.

\section{Framed correspondences and motives.}\label{sect:FrCorandMot}

In the section we recall some definitions and results from \cite{GP_MFrAlgVar} 
on the theory of framed correspondences and framed motives.

The aim of the theory 
is to make 
computations in the stable motivic homotopy category $\mathbf{SH}(k)$,
in particular to compute  hom-groups $[S, X/U]_{\mathbf{SH}(k)}$ for 
$S,X\in \Sm_k$ and an open subscheme $U\subset X$.
The fundamental Voevodksy's idea, the theory is based on, is to use for this aim the computation 
of hom-sets $[S, X/U]_{\Shv(k)}$ and 
\[[S\wedge (\PP^1/\infty)^{\wedge n}, X/U\wedge (\A^1/\Gm)^{\wedge n}]_{\Shv_\bullet(k)}\]
 in the categories of (pointed) Nisnevich sheaves $\Shv(k)$ and $\Shv_\bullet(k)$.
This computation can be done precisely and it is called as Voevodksy's lemma
\cite[lm. 3.5, prop. 3.2]{GP_MFrAlgVar}; the
the answer is called as the framed correspondences 
\cite[def. 2.5]{GP_MFrAlgVar}, and see def. \ref{def:FrCor} that follows.

At the end of the section we recall one of the main results of \cite{GP_MFrAlgVar} that gives the computation of stable motivically fibrant resolutions in positive degrees in $\mathbf{SH}(k)$ of $\Sigma_{\PP^1}^\infty X$, $X\in \Sm_k$. 

\begin{definition}\label{def:FrCor}
(i) For any schemes $S$, $X$ and closed $Y\subset X$ 
an \emph{explicit framed correspondence of a level $n$} from $S$ to the pair $(X,X-Y)$
is a set $(Z,\mathcal V,\varphi,g)$, where
$Z$ is reduced closed subscheme in $\A^n_S$ finite over $S$,
$(\mathcal V,Z)\to (\A^n_S,Z)$ is Nisnevich neighbourhood, 
$\varphi = (\varphi_i)_{i=1\dots n}\colon \mathcal V\to \A^n$ and $g\colon \mathcal V\to X$ are regular maps
such that 
$Z = (\varphi,g)^{-1}(0\times Y)$ is finite over $S$. 
So we get the diagram $$
\xymatrix{
    \A^n_S \ar[d] & \mathcal V \ar[l]_{e }\ar[r]^{(\varphi,g)} 
       & \A^n\times X\ar@{=}[r]&\A^n_X   \\
  S & Z\ar[l]_{f}\ar@{^(->}[u]\ar@{^(->}[ul]_{c} \ar[r] 
          & 0\times Y\ar@{=}[r]\ar@{^(->}[u]&0_Y\ar@{^(->}[u]
    } 
$$
where $f$ is finite, $e$ is etale, and $c$ is a closed embedding. %

(ii) Define the set $\Fr_n(S,X/(X-Y))$ of \emph{level $n$ framed correspondences}
as the \emph{set of equivalent classes} of explicit frame correspondences from $S$ to $(X,X-Y)$ under the equivalence relation given by shrinking of the neighbourhood $\mathcal V$ of $Z$.  
So 
$\varphi_1=(Z_1,\mathcal V_1,\varphi_1,g_1)$, $\varphi_2=(Z_2,\mathcal V_2,\varphi_2,g_2)$ 
are equivalent 
whenever $Z_1=Z_2=Z$ and there is an explicit framed correspondence $\varphi=(Z,\mathcal V,\varphi,g)$ with morphisms 
of Nisnevich neighbourhoods $w_i\colon \mathcal V\to \mathcal V_i$ for $i=1,2$
such that 
$w_i$ acts identically on $Z_i$,
$\varphi=w_i^*(\varphi_i)$, $g=w_i^*(g)$.

(iii) 
Define 
morphisms
\begin{gather*}\begin{array}{lll}
\Fr_n(S,X/(X-Y)) &\to& \Fr_{n+1}(S,X/(X-Y))\\
(Z,\mathcal V,\varphi,g) &\mapsto &
(Z^\prime,\mathcal V^\prime,\varphi^\prime,g^\prime), 
\end{array}\\
\begin{array}{c}
Z^\prime=Z\times 0\subset \A^{n+1}_S, \, \mathcal V^\prime=\mathcal V\times \A^1, 
\varphi_{n+1}=t_{n+1}, \, \varphi^\prime_i=\varphi_i\circ pr, \,  g^\prime=g\circ pr 
\end{array}\end{gather*}
where $t_{n+1}$ is the last coordinate function on $\A^{n+1}_S$, and $pr\colon \mathcal V\times \A^1\to \mathcal V$ is the canonical projection.

Define $\Fr(S,X/(X-Y))=\varinjlim\limits_n \Fr_n(S,X/(X-Y)),$

Define $Fr_n(S,X)=Fr_n(S,X/\emptyset)$, $Fr(S,X)=Fr(S,X/\emptyset)$.
\end{definition}

Define a category of \emph{smooth open pairs} $\Sm^\mathrm{pair}_k$ 
with objects being pairs $(X,U)$ given by a smooth scheme $X$ and open subscheme $U\subset X$, 
and with morphisms $(X,U)\to (X^\prime,U^\prime)$ given by regular maps
$f\colon X\to X^\prime$, 
$f^{-1}(U^\prime)\supset U$.
We denote an object $(X,U)\in\Sm^\mathrm{pair}_k$ also as $X/U$. 

\begin{definition}\label{def:Fr(X/U)}
For a smooth pair $(X,U)\in \Sm^\mathrm{pair}_k$ we
denote by $\Fr(X/U)$ 
the pointed sheaf $\Fr(- , X/U )$ 
pointed at the framed correspondence with empty support.
We consider $\Fr(X/U)$ as a pointed simplicial sheaf constant in the simplicial direction. 

\end{definition}
\begin{definition}
Define the functor $C^*\colon \sShv_\bullet\to \sShv_\bullet\colon F\mapsto F(-\times \Delta^\bullet)$,
where the right side is considered as the simplicial sheaf in view of the totalisation functor form the category of bi-simplicial sheaves to $\sShv_\bullet$.
Here $\Delta^\bullet$ is the standard affine co-simplicial scheme with $\Delta^n=\{(x_0,\dots x_n)\in \A^{n+1}| x_1+\dots x_n=1\}$. 
\end{definition}

Denote by $T$ the pair $(\A^1,\Gm)\in \Sm^\mathrm{pair}_k$. 
Define the smash-product functor 
\begin{equation}\label{eq:(X1,U1)wedge(X2,U2)}\begin{array}{llll}
\wedge\colon& \Sm^\mathrm{pair}_k\times \Sm^\mathrm{pair}_k &\to& \Sm^\mathrm{pair}_k \\
& ( (X_1,U_1) , (X_2,U_2) ) &\mapsto& (X_1\times X_2, X_1\times U_2\cup X_2\times U_1).
\end{array}\end{equation}
We write $T^n$ for $T^{\wedge n}$, so
$T^n = (\A^n,\A^n-0)\in \Sm^\mathrm{pair}_k$. 
\begin{definition}\label{def:MPfrMP}
Define the $\PP^{\wedge 1}$-spectra 
\begin{gather}
\MPfr(X,U) = ( \Fr(X/U), \Fr( (X/U) \wedge T), \dots \Fr( (X/U) \wedge T^i), \dots ) \label{eq:sec1:MPfr}
\\
\MP(X,U) = C^*(\MPfr(X,U)), 
\end{gather}
see definition \ref{def:MTfr} in Appendix for the structure maps of the spectra $\MPfr(X,U)$.
Define $\MPfr(X)=\MPfr(X,\emptyset)$, $\MP(X)=\MP(X,\emptyset)$.
\end{definition}
\begin{remark}
Writing $(X/U) \wedge T$ 
in \eqref{eq:sec1:MPfr} we mean $(X,U)\wedge T$ 
in sense of \eqref{eq:(X1,U1)wedge(X2,U2)}.
\end{remark}

Denote by $(-)_f $ the  the NIsnevich local resolution endo-functor on $\sShv_\bullet$.
Now we can formulate one of the main results of \cite{GP_MFrAlgVar}.
\begin{theorem}[Theorem 4.1 \cite{GP_MFrAlgVar}]
Let $X$ be a smooth scheme over a infinite perfect field $k$.
Then the canonical morphism 
\[\Sigma^\infty_{\PP^1}(X)\to \MP(X)_f\]
is stable motivic weak equivalence and the right side is a motivically fibrant $\Omega_{\PP^1}$-spectrum in positive degrees.
\end{theorem}
\begin{remark}
In other words the last statement of the theorem above means that
for each $l>0$ the pointed sheaf $C^*\Fr(X\wedge T^l)_f$ is motivically fibrant within the Morel-Voevodsky model structure on $\sShv_\bullet$, and the canonical morphism
\[C^*\Fr(X\wedge T^l)_f\to \Omega_{\PP^1}(C^*\Fr(X\wedge T^{l+1})_f)
,\]
see def. \ref{def:MPfr}, is a schemewise simplicial weak equivalence.
\end{remark}

\section{Quasi-finite framed correspondences}\label{sect:QuasiFiniteCor}

In the section we recall the definition of quasi-finite framed correspondences $\Frqf(-,X/U)$ introduced in \cite{GNP_FrMotiveRelSphere},
that is a useful tool in studying of framed correspondences. 

Main result of the section is the Nisnevich local equivalence $\Fr(X//U)\simeq\Frqf(X/U)$, corollary \ref{cor:FrPairFreqNisEq}, where $X//U$ is simplicial cone.
This generalises \cite[corollary 5.3]{GNP_FrMotiveRelSphere}.
We apply this to prove the Mayer-Vietoris property for the framed correspondences functor $\Fr\colon \Sm_k\to \Shv_\bullet$, corollary \ref{cor:MayerVietoris}, and see defs. \ref{def:FrCor}-\ref{def:Fr(X/U)} for the functor $\Fr$.

\begin{definition}\label{def:Frqf}
Let $\Fr_n^\mathrm{qf}(-, X/(X-Y))$ be a subpresheaf of $\Fr_n(-, X/(X-Y) )$ that consists of
framed correspondences $a=(\mathcal V,Z,\varphi,g)$ such that
$\varphi^{-1}(0)$ is quasi-finite over $U$. Set $\Frqf(-,X/(X-Y))=\varinjlim_n \Frqf_n(-,X/(X-Y))$.
\end{definition}

\begin{proposition}\label{prop:FrX/FrUeqnisFrPair}
Let $X\in \Sm_k$ and $U\subset X$ be an open subscheme.
Then the natural morphism of pointed Nisnevich sheaves
\[\Fr_n(X)/\Fr_n(U)\to \Fr^{\qf}_n(X/U)\]
is a Nisnevich local equivalence.
\end{proposition}
\begin{proof}
To prove the claim 
we consider a henselian $k$-schemes $S$ and construct the inverse morphism for the canonical morphism of germs
\begin{equation}\label{eq:Xfr/Ufr->(X,U)frqf}\begin{array}{llll}
c\colon& \Fr(S,X)/\Fr(S,U) & \to & \Fr^\qf(S,X/U) \\
& (Z, V, \varphi, g) & \mapsto & (Z\times_{g,X,i} (X\setminus U), V, \varphi, g)
.\end{array}
\end{equation}

Let $(Z, V, \varphi, g)\in \Fr^\qf(S,X/U)$.
Then $g\colon V \to X$ and $Z=V\times_{(\varphi,g), X, i}(X\setminus U)$,
where $i\colon X\setminus U\to X$ is the canonical closed immersion.
Since $S$ is local henselian it follows by lemma \ref{lm:qfhenseliansplit} that 
$Z(\varphi) = W_f\amalg W_{qf}$, 
such that $Z\subset W_f$ and $W_f$ is finite over $S$.
Define the morphism of pointed sets
\begin{equation}\label{eq:Xfr/Ufr=(X,U)frqf}\begin{array}{llll}
r \colon & \Fr^\qf(S,X/U) & \rightarrow & \Fr(S,X)/\Fr(S,U) \\
& (Z, V, \varphi, g) &\mapsto & (W_f, V- W_{qf}, \varphi, g)
.\end{array}
\end{equation}
We need to check that the morphism $r$ is well defined. 
To do this we need to compare the set of etale neighbourhoods of the subscheme $Z$ and the subscheme $W_f$.
The sets are equal that follows form lemma \ref{lm:NeighbourhoodsFinsupports}. 

Now we see that
the composition $r\circ c$ is identity immediately by definitions.
The composition $c\circ r$ is identity since 
by definition of $\Fr^\qf$
the element given by 
$(Z, V, \varphi, g)$ 
is equal to the element given by 
$(Z, V- W_{qf}, \varphi, g)$.
\end{proof}
\begin{lemma}\label{lm:qfhenseliansplit}
Let $S$ be a henselian local scheme and $W$ be a quasi-finite scheme over $S$.
Then $W=W_f\amalg W_{qf}$, where $W_f$ is finite over $S$ and the closed fibre of $W_{qf}$ is empty.
\end{lemma}
\begin{proof}
By Zariski's main \cite[Theorem~8.12.6]{GD67}
the morphism $W\to S$ can be passed throw
$W\to \overline W\to S$ with $\overline W$ being finite over $S$.
Then the closed fibre 
of $\overline W$ splits 
\[\overline W\times_S \sigma = W\times_S \sigma \amalg (\overline W\setminus W)\times_S \sigma.\]
Hence since $U$ is local henselian the scheme $\overline W$ splits into the union of clopen subschemes
\[\overline W = W_f \amalg \overline W_{qf}, \,W_f\times_S \sigma = W\times_S \sigma, \,\overline W_{qf}\times_S \sigma = (\overline W\setminus W)\times_S \sigma.\]
Thus we get the required splitting 
\[W = W_f\amalg W_qf, \,W_f\times_S \sigma = W, \,W_{qf}\times_S \sigma = \emptyset.\]
\end{proof}
\begin{lemma}\label{lm:NeighbourhoodsFinsupports}
Let $W\to S$ be a finite morphism, $S$ be a local henselian scheme, $\sigma\in S$ be the closed point.
Let $W^\prime\to W$ be en etale morphism such that $W^\prime\times_S \sigma \simeq W\times_S \sigma$.
Then $W^\prime\simeq W$.
\end{lemma}
\begin{proof}
Since $S$ is local henselian and $W\to S$ is finite, it follows that $W$ is semilocal henselian. So the claim follows. 
\end{proof}

\begin{definition}\label{def:SimplCone}
Let $Y\hookrightarrow X$ be a morphism 
of smooth schemes. 
The \emph{simplicial cone} $X//Y$
of the morphism $Y\rightarrow X$ is a simplicial scheme defined as follows.

If $U=\emptyset$ then $X//Y=X$.
If $U\neq\emptyset$,
then $X//Y$ is the colimit of the diagram in the category of simplicial schemes
\begin{equation}\label{eq:SimplCone}
X//U=\colim(X \xleftarrow{j} U \xrightarrow{i_1} \Delta^1_s\times U\xleftarrow{i_0} U\rightarrow * )\in \Delta^\mathrm{op}\mathrm{Sm_k}
,\end{equation} 
where
$\Delta^1_s$ is the simplicial interval,
$i_0,i_1\colon U\to \Delta^1_s\times U$ are the unit and zero faces,
and $*$ is a rational point that is the base point of $X//U$.
\end{definition}

\begin{corollary}\label{cor:FrPairFreqNisEq}
Let $X\in \Sm_k$ and $U\subset X$ be an open subscheme.
Then the natural morphism of simplicial pointed sheaves 
\[\Fr_n(X//U)\to \Fr^{\qf}_n(X/U)\] is a Nisnevich local equivalence,
see def. \ref{def:SimplCone} 
and def. \ref{def:FrcalX/calU} for $\Fr(X//U)$.
\end{corollary}
\begin{proof}
The claim follows by proposition \ref{prop:FrX/FrUeqnisFrPair}
and the simplicial equivalence of pointed simplicial presheaves $\Fr_n(X//U)\simeq \Fr_n(X)/\Fr_n(U)$.
\end{proof}

\begin{corollary}\label{cor:MayerVietoris}
Let $X\in \Sm_k$, $U,V\subset X$ be open subschemes, $X=U\cup V$.
Then the canonical morphisms of pointed simplicial Nisnevich sheaves 
\[\Fr( V//(U\cap V))\to \Fr(X/U), \, \Fr( V \amalg^\sim_{(U\cap V)} U ) \to \Fr( U \cup V )\]
are Nisnevich local equivalences;
see 
def. \ref{def:hocoeq} and 
def. \ref{def:FrcalX/calU}
for $\Fr(V \amalg^\sim_{(U\cap V)} U)$. 
\end{corollary}
\begin{proof}[Proof]
It is easy to see from the definition that 
$\Frqf(V/(U\cap V))\to \Frqf(X/U)$ 
Then by corollary \ref{cor:FrPairFreqNisEq} we have the Nisnevich local equivalence
$\Fr(V//(U\cap V))\to \Fr(X//U)$.
Then the equivalence 
$\Fr( V\amalg^\sim_{(U\cap V)} U )\to \Fr(X)$
follows.
\end{proof}
\begin{definition}\label{hocoeq}
Let $f_1\colon U\to X$ and $f_2\colon U\to Y$ be morphisms in $\Sm_k$. 
Define a simplicial scheme
$V\amalg^\sim_{(U\cap V)} U$ as 
the colimit in the category of simplicial schemes
\[
V\amalg^\sim_{(U\cap V)} U = 
\colim( X \xleftarrow{f_1} U \xrightarrow{i_1} U\times \Delta^1 \xleftarrow{i_0} U \xrightarrow{f_2} Y )\in \Delta^\mathrm{op}\Sm_k
\] 
where 
$\Delta^1$ denotes the simplicial interval, and,
and the morphisms $i_0,i_1\colon U\to U\times\Delta^1$ are the face morphisms.

\end{definition}
\begin{remark}\label{def:hocoeq}
The simplicial scheme $V\amalg^\sim_{(U\cap V)} U$
represents the homotopy co-equaliser $\hocoeq( U\to (X\amalg Y) )$ in the homotopy category of simplicial presheaves.
\end{remark}

\section{Moving lemma.}\label{sect:MovingLemma}

In this section 
we prove a moving lemma 
contracting 
the sheaf $\Fr(-,X/(X-Y))$ to the subsheaf $\Frqf(-,X/(X-Y))$
by an exhaustive family of partly defined $\A^1$-homotopies.
Throw out the section we work with given 
smooth affine $X\in \AffSm_k$ over an infinite field $k$, 
and a closed subscheme $Y\subset X$.

The main results of the section are proposition \ref{pr:AlmAllQF} and proposition \ref{prop:hom-to-qf}. 
Before we formulate the result let us give the following definition.

\begin{definition}
Given a scheme $Y$ and integers $n$ and $i$, 
we put 
$$\Gamma_d=\{s=(s_i)\in \Gamma(\mathbb P^n\times X,\mathcal O(d)^n)| s_i\big|_{\mathcal N\times Y}=x_i \cdot x_\infty^{d-1}\},$$
where $\mathcal N = Spec\,k[\A^n]/(x_1,\dots x_n)^2\subset Y\times\mathbb P^n$.
\end{definition}

\begin{proposition}\label{pr:AlmAllQF}
Let $U$ be a scheme of a finite Krull dimension, and $\varphi=(Z,\mathcal V,\varphi,g)\in \Fr_n(U,X/(X-Y))$. 
Then for any $l\in \mathbb Z$ for some $d\in\mathbb Z$ there is an open subset $\mathcal U\subset \Gamma_d$ 
such that \begin{itemize}
\item[(1)] $\codim( W / \Gamma_d ) \geq l$, where $W=(\Gamma_d\setminus \mathcal U)\subset \Gamma_d$ is the closed complement,
and 
\item[(2)] for any rational point $s=(s_i)\in \mathcal U$ 
the scheme-theoretical preimage $(f\circ (\varphi,g))^{-1}(0)$ is quasi-finite over $U$, 
where $f=(f_i)$, $f_i=s_i/x_\infty^d$.
\end{itemize}
Moreover if $\varphi\in \Frqf_n(U,X/(X-Y))$ then there is $\mathcal U\subset \Gamma_d$ as above and such that $\mathcal U\ni (x_1x^{d-1}_\infty,x_1x^{d-1}_\infty,\dots  x_n x^{d-1}_\infty)$. 
\end{proposition}

\begin{proof}

Set $\mathcal N = \Spec k[\A^n]/(x_1,\dots x_n)^2$, and for any scheme $G$ define 
\begin{equation}\label{eq:GammadG}\Gamma_{d,G} = \{ s=(s_1,\dots ,s_n)\in \Gamma(\mathbb P^n\times X\times G, \mathcal O(d))^n  |
s_i\big|_{\mathcal N\times\overline{Y}\times G} = x_i\cdot x_0^{d-1}
\}.\end{equation}
Consider the universal section $\tilde s=(\tilde s_i)\in \Gamma_{d,\Gamma_d}$, 
and a closed subscheme $S=(\tilde f\circ(\varphi,g))^{-1}(0)\subset \mathcal V\times\Gamma_d$,
where $\tilde f=(\tilde f_i)$, $\tilde f_i=\tilde s_i/x_\infty^d$. 
In other words \[S= \{(p,s)\in \mathcal V\times\Gamma_d \colon s((\varphi,g)(p))=0\}.\]

Let $B\subset U\times \Gamma_d$ be a closed subscheme such that $z=(u,s)\in B$ whenever the fibre $S\times_U z$ is not quasi-finite over $z$.
Define $\mathcal U=\Gamma_d\setminus \overline{\pr_{\Gamma_d}( B)}$, where $\overline{\pr_{\Gamma_d}( B) }$ denotes the closure of the image of $B$ under the projection $\pr_{\Gamma_d}\colon U\to \Gamma_d$.
The proposition follows form the lemma.
\begin{lemma}\label{lm:codimcalUGammad}
For any $m\in \mathbb Z$ there is $M(m)$, for all $d>M(m)$, 
we have \begin{equation}\label{eq:codimcalUGammad}\codim( (\Gamma_d\setminus \mathcal U) / {\Gamma_d} )\geq m-\dim_k U,\end{equation}
where $\mathcal U\subset \Gamma_d$ is the open subscheme as above.
\end{lemma}
Let us deduce the proposition.
Choose some $m>dim\,U+l$, and $d>M(m)$ as in the lemma above.
Then point (1) of the proposition follows by \eqref{eq:codimcalUGammad}.
Point (2) follows by the definition of $B$,
and moreover, it follows 
that if $\varphi\in \Frqf_n(U,X/(X-Y))$ then $B\not\ni (x_1 x_\infty^{d-1}, x_2 x_\infty^{d-1}, \dots x_n x_\infty^{d-1})$.

Briefly speaking we argument for lemma \ref{lm:codimcalUGammad} is the following. Consider the $m$-th power $S^m_{U\times\Gamma_d}$ of $S$ over $U\times\Gamma_d$. For large enough $d$ the closed subscheme $S\subset \mathcal V\times\Gamma_d$ is defined by $n$ independent equations, and 
$S^m_{U\times\Gamma_d}\subset \mathcal V^m_U\times\Gamma_d$ is defined by $n m$ independent equations. 
So $\dim (S^m_{U\times\Gamma_d}) = \dim (\mathcal V^m_U\times\Gamma_d)-nm = \dim U + \dim \Gamma_d$. 
On other side the relative dimension of $S$ over $B$ is at least one, and $\dim_{B} (S^m_{U\times\Gamma_d})\geq m$. 
Then $\codim (B/ U\times\Gamma_d)\geq m$, and the claim follows. 
We start the strict detailed proof.

\begin{proof}[Proof of lemma \ref{lm:codimcalUGammad}]


Consider the schemes $\PP^n\times X$,
$\mathcal E = \A^n\times{X}\times U$ and $\overline{\mathcal E} = \PP^n\times {X}\times U$.
The inverse images of $\mathcal O(1)$ from $\PP^n$ to $\mathbb P^n\times X$ and $\overline{\mathcal E}$ 
we denote these sheaves by the same symbol.
Let $[x_0\colon x_1\colon \dots x_n]$ denote coordinates on $\PP^n$, and their inverse images as well.
Since $X$ is affine, it follows that $\mathcal O(1)$ on $\mathbb P^n\times X$ is ample.

Next, define the regular map $$
\psi = (\varphi,g,\pr^{\mathcal V}_U)\colon \mathcal V\to \mathcal E , \text{ where } \pr^{\mathcal V}_U\colon \mathcal V\to U.
$$
Since $Z=\psi^{-1}(0\times Y\times U)$ is finite over $U$, it follows that $\psi$ is quasi-finite over $0\times Y\times U$.
Hence there is a Zariski neighbourhood $\mathcal V^\prime$ of $0\times Y\times U$ in $\mathcal E$
such that $\psi$ is quasi-finite over $\mathcal V^\prime$.
Shrink $\mathcal V$ to the open subscheme $\mathcal V\times_{\mathcal E} \mathcal V^\prime\subset \mathcal V$.
Then new $\psi$ is quasi-finite.

Let $\psi^m\colon \mathcal V^m_U\to \mathcal E^m_U $
be $m$-th power of $\psi$,
and define open subschemes
\begin{multline*}
\cEm =
\{ (p_1,\dots, p_m)\in \mathcal E^m_U |\,
p_i\neq p_j, \text{ for } i\neq j, \text{ and }\\
p_i\not\in 0\times Y \times U \text{ for all }i
\}\subset\mathcal E^m_U,
\end{multline*} and
\begin{equation*}
\cVm = 
{(\psi^m)}^{-1}(\cEm) \,=\,
\mathcal V^m_U - {(\psi^m)}^{-1}( \mathcal E^m_U \setminus \cEm ) \subset
\mathcal V^m_U
.\end{equation*}
  
Define a closed subscheme
\[\cSm\subset \{(p,s)\in \cVm\times \Gamma \colon s((\varphi,g)(p))=0\}\subset \cVm\times\Gamma.\]
Note that $\cSm= S^m_{U\times\Gamma_d}\times_{\mathcal V^m_U} \cVm\subset S^m_{U\times\Gamma_d}$ is an open subscheme.
Consider the commutative diagram of schemes over $U$
$$\xymatrix{
& \cVm\times \Gamma_d\ar[ld]\ar[rd] & \\
U\times\Gamma_d & \cSm\ar[l]\ar[r]\ar@{^(->}[u] & \cVm  \\
B\ar[u] & BS\ar[u]\ar[l] & \\
}$$
where the left bottom square is Cartesian.

It follows by sublemma \ref{lm:indeppointcond}, which follows in the text, that 
$\codim_{\cVm\times\Gamma_d} \cSm = nm$, hence
\[\dim_{U} ( \cSm ) = \dim_U(\cVm\times\Gamma_d) - nm=\dim_k \Gamma_d,\] 
and thus
\begin{equation}\label{eq:codimcSmGammad}\dim ( \cSm / (U\times \Gamma_d) ) = 0.\end{equation} 

On other side 
the points of the subscheme $B$ are 
given by the pairs $(u,s)$, $s\in \Gamma(\overline{\mathcal E},\mathcal O(d)^n)$, $u\in U$, 
such that the vanishing locus $Z((\varphi,g)^*(s))=(f\circ (\varphi,g))^{-1}(0)$ is not quasi-finite over $(u,s)$.
So for any $(u,s)\in B$ we have $\dim_{(u,s)} Z((\varphi,g)^*(s))\geq 1$,
and thus $\dim ( (S\times B) / B)\geq 1$. 
Then
\begin{equation}\label{eq:codimSB/B}\codim ( BS / B )\geq m,\, BS =\cSm \times_{(U\times\Gamma_d)} B.\end{equation}

Thus 
\begin{multline*}
\dim \overline{\pr(B)}\leq \dim B\stackrel{\eqref{eq:codimSB/B}}{\leq} \dim BS-m\leq\dim \cSm - m\stackrel{\eqref{eq:codimcSmGammad}}{\leq} \\ \dim(U\times\Gamma_d)-m=\dim\Gamma_d-(m-\dim U).\end{multline*}

\begin{sublemma}\label{lm:indeppointcond}
Given a point $p=(p_1,\dots p_m)\in \cEm$, 
a closed subscheme $Q\subset \mathbb P^n\times X\times  p$ is the union of graphs of $p$-points $(p_i)\colon p\to\mathbb P^n\times X$.
Let 
\[r_p\colon \Gamma_{d,p}\to \mathcal O(Q)^n \simeq k(p)^{m n}\colon (s_i)\mapsto ( (s_i/x_\infty^d)\big|_{Q} )\]
denotes the restriction homomorphism, see \eqref{eq:GammadG} for $\Gamma_{d,p}$. 

Then 
for any $m\in \mathbb Z$ there is $M(m)\in\mathbb Z$ such that, for all $d > M(m)$, for any 
point $p\in \cEm$, the homomorphism $r_p$ is surjective.
\end{sublemma}
\begin{proof}[Proof of lemma \ref{lm:indeppointcond}.] 
For any $p=(p_1,\dots p_m)\in \cEm$ and $d\in \mathbb Z$,
denote 
\[j_p\colon Q\amalg \mathcal N\times \overline Y\to \mathbb P^n\times X\times p.\] 
Then we have 
the restriction homomorphism of sheaves 
\[\mathcal O(d)^n\twoheadrightarrow {j_p}^*({j_p}_*(\mathcal O(d)^n))\] on $\mathbb P^n\times X\times p$
and
the restriction homomorphism of global sections 
\[\Gamma(\overline{\mathcal E},\mathcal O(d))^n\to \Gamma(Q\amalg \mathcal N\times\overline Y)^n.\]


Let
\[j\colon G\to \overline{\mathcal E}, \;
G = \left(\coprod\limits_{i=1\dots m} \Delta_{i} \right)
  \coprod \mathcal N\times Y\times \cEm,\]
and $\Delta_i\subset \overline{\mathcal E}\times_U \cEm$ denotes graph of the $i$-th projection $\cEm\to \mathcal E$. 

The scheme $\mathbb P^n\times X\times p$ is equal to $\overline{\mathcal E}\times_{U} p$, and $j_p$ is the fibre of $j$.
Consider the universal restriction homomorphism of coherent sheaves over $\cEm$
\[\rho_{\cEm}\colon \mathcal O(d)^n \twoheadrightarrow j_*(j^*(\mathcal O(d)^n)), \]
that is the surjective homomorphism of coherent sheaves on the scheme $\overline{\mathcal E}\times_U \cEm$ that is open subscheme in $\mathbb P^n \times U\times (\mathbb A^n\times X)^m $.
Consider the direct image of $\rho_{\cEm}$ 
\[\pr_*(\rho_{\cEm})\colon {\pr}_*(\mathcal O(d)^n) \to {\pr}_*(j_*(j^*(\mathcal O(d)^n))), \pr\colon {\overline{\mathcal E}}\times_U \cEm\to \cEm, \]
that is a homomorphism of coherent sheaves on $\cEm$.
It follows by the relative version of Serre's theorem on ample bundles
\cite[Chapter~III, Theorem~8.8]{Ha77}]
that for large enough $d$ the homomorphism $\pr_*(\rho_{\cEm})$ is surjective.

The sheaf $\pr_*(\mathcal O(d^n))$ 
is the constant sheaf on $\cEm$ 
defined by to the $k$-vector space
$\Gamma(\mathbb P^n\times X, \mathcal O(d) )^n$,
and the sheaf $\pr_*(j_*(j^*(\mathcal O(d)^n)))$ is equal to
\[
\bigoplus\limits_{i=1\dots m}  \mathcal O(\Delta_i)^n \oplus 
{\Gamma( \mathcal N\times Y, \mathcal O(d) )^n } 
\simeq
\mathcal O(\cEm)^{n m} \oplus 
{\Gamma( \mathcal N\times Y, \mathcal O(d) )^n } 
,\] 
where the second summands denote the constant sheaf on $\cEm$  
defined by the $k$-vector space ${\Gamma( \mathcal N\times Y, \mathcal O(d) )^n }$.

Hence since $\pr_*(\rho_{\cEm})$ is surjective the fibre of $\pr_*(\rho_{\cEm})$ at a point $p=(p_1,\dots p_m)\in\cEm$
is equal to the surjective homomorphism of $k(p)$-vector spaces
\[
\Gamma( \mathbb P^n\times X ,\mathcal O(d) )^n\otimes k(p) \twoheadrightarrow
k(p)^{m n} \oplus 
\Gamma(\mathcal N\times {Y}, \mathcal O(d))^n\otimes k(p), \text{ for any }p\in\cEm.
\]
Thus 
the homomorphism 
$\Gamma_{d,p}=\Gamma_d\otimes k(p)\to k(p)^{n m}$ is surjective for any $p\in \cEm$,
because of the 
diagram of pointed sets with Cartesian squares
$$\xymatrix{
(\Gamma_{d,p}, (x)) \ar@{^(->}[d] \ar@{->>}[r] & ( \bigoplus\limits_{i=1\dots m} k(p_i) , (x) )  \ar@{^(->}[d] \\
(\Gamma(\mathbb P^n\times X\times p, \mathcal O(d))^n, (x) ) \ar@{->>}[d]\ar@{->>}[r] & 
  (\bigoplus\limits_{i=1}^n  k(p_i), (x) ) \times 
  (\Gamma(\mathcal N\times Y, \mathcal O(d))^n_{k(p)}, (x)) \ar@{->>}[d] 
\\
(\Gamma(\mathcal N\times Y, \mathcal O(d))^n_{k(p)}, (x)) \ar@{=}[r] &
(\Gamma(\mathcal N\times Y, \mathcal O(d))^n_{k(p)} , (x) )
,}$$
where $\Gamma(\mathcal N\times Y, \mathcal O(d))^n_{k(p)}=\Gamma(\mathcal N\times Y, \mathcal O(d))^n\otimes k(p)$, and
all sets of sections are pointed at the class of the vector-section $(x)=(x_i \cdot x_\infty^{d-1})$ given by coordinates on $\PP^n$.
\end{proof}
\end{proof}
\end{proof}

\begin{definition}\label{def:Phi^s}
Suppose $U,X\in \Sm_k$, $Y\subset X$ is a closed subset, 
$\varphi=[(Z,\mathcal V,\varphi,g)]\in Fr_n(U,X/(X-Y))$ is a framed correspondence, and 
$s=(s_i)_{i=1\dots n}\in \Gamma(\mathbb P^n\times  X ,\mathcal O(d))^n$ is a section
such that $s_i\big|_{\mathcal I(0\times Y)^2}=x_i$. 

Denote by  $\varphi^{s}\in Fr_n(U,X/(X-Y))$ the framed correspondence defined by the class of the set $(Z,\mathcal V,\zeta \circ \varphi,g)$, where $\zeta\colon \A^n_X\to \A^n_X$ is the regular map defined by functions $s_i/x_{\infty}^d\in k[\A^n_X]$, $i=1\dots n$.
\end{definition}
\begin{definition}\label{def:Phi^lambdas}
Suppose $U$, $X$, $Y$ are as in def. \ref{def:Phi^s}, and 
$\varphi=[(Z,\mathcal V,\varphi,g)]\in Fr_n(U\times\A^1,X/(X-Y))$ is a framed correspondence, and 
$s=(s_i)_{i=1\dots n}\in \Gamma(\A^1\times \mathbb P^n\times  X ,\mathcal O(d))^n$ is a section
such that $s_i\big|_{\mathcal I(\A^1\times 0\times Y)^2}=x_i$. 

Denote by  $\varphi^{s}\in Fr_n(U\times \A^1,X/(X-Y))$ the framed correspondence given by 
$(Z,\mathcal V,\zeta \circ \varphi^\prime,g)$, 
where $\varphi^\prime = \zeta_\lambda\circ \varphi$, 
$\lambda\colon \mathcal V\to \A^1$ is the canonical projection,
and for any $\lambda \in \A^1$, the map $\zeta_\lambda\colon \A^n_X\to \A^n_X$ is the regular map 
defined by functions $(s_i\big|_{\lambda\times\PP^n})/x_{\infty}^d\in k[\A^n_X]$, $i=1\dots n$.
\end{definition}

\begin{proposition}\label{prop:hom-to-qf}
Assume the field $k$ is infinite, and let $U\in \Sm_k$.
Then for any framed correspondence 
\[ {(a) }\, \varphi\in \Fr_n(U,X/(X-Y)), \text{ or } { (b) }\, \varphi\in \Frqf_n(U,X/(X-Y)),\]
there is 
$d\in \mathbb Z_{>0}$ and
a vector of sections $s=(s_i)\in\Gamma(\mathbb P^n_X,\mathcal O(d))^n$ such that 
\[ {(a)}\, \varphi^s\in \Frqf_n(U, X/(X-Y) ), \text{ or } { (b) }\, \varphi^{\lambda s+(1-\lambda) x}\in  \Frqf_n(U\times\A^1, X/(X-Y) )\]
respectively, where $x=(x_i x_\infty^{d-1})$.

\end{proposition}\begin{proof}
a)
Since $k$ is infinite the claim follows immediate from proposition \ref{pr:AlmAllQF}.

b)
By proposition \ref{pr:AlmAllQF} for some sufficiently big $d$ 
there is an open subscheme $\mathcal U\subset \Gamma_d$ such that 
$\codim (\Gamma_d\setminus \mathcal U)\geq 2$ and $\varphi^{s}\in Fr^{qf}_n(U\times\A^1, X/(X-Y) )$ for any $s\in \mathcal U$.
Hence we have  
$\varphi^{\lambda s+(1-\lambda) x}\in Fr^{qf}_n(U\times\A^1, X/(X-Y) )$,
for any $s\in \Gamma_d - Cl_{\Gamma_d}({pr_x^{-1}(pr_x(\Gamma_d\setminus \mathcal U))})$,
where $pr_x\colon \Gamma_d-x\to \mathbb P^{\dim\Gamma_d-1}$ is the linear projection with centre $x$.

Whence 
$\codim Cl_{\Gamma_d}({pr_x^{-1}(pr_x(\Gamma_d\setminus \mathcal U))})\geq \codim (\Gamma_d\setminus \mathcal U) -1\geq 1$, and so the claim follows.

\end{proof}

\section{Proof of the result}
\label{sect:ProofofTheorems}

In the section we prove the Cone Theorem announced in the introduction, theorem \ref{th:ConeTh}, and apply it to deduce the formula for the motivically fibrant resolution in $\mathbf{SH}(k)$, corollary \ref{cor:MPmotresolution}.
Throw out the section the base filed $k$ is assumed being infinite, and it is perfect starting form corollary \ref{cor:MPmotresolution}.
\begin{theorem}\label{th:ConeTh}
Let $X$ be affine smooth $k$-scheme, and $U\subset X$ be open subscheme. 
The natural morphism of pointed simplicial sheaves 
\[\Fr(X//U)\to \Fr(X/U)\] is a 
motivic equivalence,
see def. \ref{def:SimplCone} for $X//U$.
\end{theorem}
\begin{proof} 
By proposition  \ref{prop:FrqfFrAEq}, that follows in the section, we have the simplicial schemewise equivalence of pointed simplicial sheaves $C^*\Fr^{qf}(X/U)\to C^*\Fr(X/U)$.
By corollary \ref{cor:FrPairFreqNisEq} we have the schemewise equivalence of pointed simplicial sheaves $\Fr(X//U)_f\to \Frqf(X/U)_f$.
So the claim follows.
\end{proof}

\begin{proposition}\label{prop:FrqfFrAEq}
Let $X$ be affine smooth $k$-scheme, and $U\subset X$ be open subscheme. 
Then the natural morphism of sheaves of pointed sets 
\[\Frqf(X/U)\to \Fr(X/U)\] is an $\A^1$-equivalence.
\end{proposition}

\begin{definition}
Let $F$ be a sheaf of pointed sets.
For a variety $X$ denote by $F\times X$
the (naive) schemewise product of sheaves,
$(F\times X)(S) = F(S)\times \mathrm{Map}(S,X)$,$S\in \Sm_k$.
In particular denote by $F\times \A^1$
the presheaf given by
$$(F\times \A^1)(S) = F(S)\times \mathcal O(S), S\in \Sm_k.$$ 
\end{definition}
Let
$F=F_\infty\supset \dots F_l\supset F_{l-1}\supset\dots \supset F_0$
be a filtration of the presheaf $F$ and denote by $i_l\colon F_{l}\to F_{l+1}$ the canonical inclusion.
Define a telescope of the filtration $F_*$ as the 
pointed presheaf 
\begin{equation}\label{eq:F_Filtration}
\mathrm{Tel}(F_*)=
\dots \amalg (F_{l+1}\times \A^1)\amalg_{F_{l}} (F_{l}\times \A^1) \amalg_{F_{l-1}}\amalg\dots\amalg_{F_0}(F_0\times \A^1)
\end{equation}
where the coproducts are defined with respect to 
the inclusions 
$i_l\times 1\colon F_l\times 1\to F_{l+1}\times\A^1$,
$id_{F_l}\times 0\colon F_l\times 0\to F_{l}\times\A^1$.

\begin{lemma}
For any pointed presheaf $F$ with the filtration \eqref{eq:F_Filtration} the canonical morphism
$\mathrm{Tel}(F)\to F$
is an $\A^1$-homotopy equivalence
\end{lemma}
\begin{proof}
Clearly for any presheaf $F$ the canonical morphism $F\times \A^1\to F$ is an $\A^1$-equivalence.
Denote
$\mathrm{Tel}^l(F) = 
(F_{l}\times \A^1)\amalg_{F_{l-1}} (F_{l-1}\times \A^1) \amalg_{F_{l-2}}\amalg\dots\amalg_{F_0}(F_0\times \A^1)
.$
It follows that the canonical morphism $\mathrm{Tel}^l(F)\to F$ is an $\A^1$-equivalence.

Since 
the injective limits preserves weak equivalences on simplicial sets, and the functor $\underline{Hom}(\Delta^\bullet,-)$ commutes with the injective limits, it follows that injective limits preserve $\A^1$-equivalences.
Thus the claim follows since 
$F=\varinjlim_l F_l$ and $\mathrm{Tel}(F_*)=\varinjlim_l\mathrm{Tel}^l(F_*)$.
\end{proof}

We define the notion of the telescope for filtrations indexed by an arbitrary filtering
ordered set. 

Let $A$ is an ordered set and $F_\alpha\subset F$, $\alpha\in A$, be a filtration on $ A$,
i.e. 
$F=\bigcup_{\alpha\in  A} F_\alpha$
and
$F_\alpha\subset F_\beta$ for $\alpha<\beta$. 
Define a telescope of the filtration $F_*$ by the following.

Consider the ordered set $\mathcal CA$ of finite linearly ordered subsets of $A$, i.e. 
the category with objects being the sets $s=(\alpha_0>\alpha_1>\alpha_2>\dots>\alpha_n)$ 
and a unique morphism $s_1\to s_2$ if $s_1$ is a subset of $s_2$.
For any $n$-dimensional simplex $s=(\alpha_0>\alpha_1>\alpha_2>\dots>\alpha_n)$ in $\mathcal SA$
define $F_s = F_{\alpha_0}$. 
We get the functor $\mathcal CA$ to the category of subpresheaves of $F$.
Now we put
$$
\mathrm{Tel}(F_*) =\varinjlim\limits_{s\in \mathcal CA} F_s\times\Delta^{\#s},
$$
where $\Delta^n$ denote the affine $n$-dimensional simplex, 
the morphisms $F_{s_1}\times\Delta^{\#s_1}\to F_{s_2}\times\Delta^{\#s_2}$ 
are given by the product of the morphisms $F_{s_1}\to F_{s_2}$ and $\Delta^{\#s_1}\to \Delta^{\#s_2}$, and the last one is the face map corresponding to the inclusion $s_1\subset s_2$. 

\begin{lemma}\label{lm:OSetTel}
For any presheaf $F$ with the filtration \eqref{eq:F_Filtration} indexed by the filtering set $A$ the canonical morphism
$\mathrm{Tel}(F)\to F$
is an $\A^1$-homotopy equivalence.
\end{lemma}
\begin{proof}
Clearly for any presheaf $F$ the canonical morphism $F\times \Delta^n\to F$ is an $\A^1$-equivalence.
For any $\alpha\in A$ denote by
$$
\mathrm{Tel}^\alpha(F_*) =\varinjlim\limits_{s\in \mathcal CA_{<\alpha}} F_s\times\Delta^{\#s},
$$
where
$\mathcal CA_{<\alpha}$ is the category of 
$\mathcal CA$ spanned by the objects $s=(\alpha_0>\alpha_1>\dots\alpha_n)$ with $\alpha_0\leq \alpha$. 
It follows from the above that the canonical morphism $\mathrm{Tel}^\alpha(F)\to F$ is an $\A^1$-equivalence.

Since weak equivalences on simplicial sets are preserved under the injective limits, and $\underline{Hom}(\Delta^\bullet,-)$ commutes with the injective limits, it follows that injective limits preserves $\A^1$-equivalences.
Thus the claim follows from the commutative diagrams, $\alpha\in A$,
$$\xymatrix{
\mathrm{Tel}^{\alpha}(F)\ar[d]\ar[r] & F_\alpha\ar[d]\\
\mathrm{Tel}(F)\ar[r] & F
}$$
since 
$\mathrm{Tel}(F)=\varinjlim_\alpha \mathrm{Tel}^{\alpha}(F)$,
$F = \varinjlim_\alpha F_\alpha$.
\end{proof}

\begin{lemma}\label{lm:agrreedPhi^s_maps}
For any finite set of sections $\alpha\subset \Fr$ there is a section $s\in \Gamma(\mathbb P^n_X,\mathcal O(d))^n$
such that
\[\alpha^{s_{\alpha}}\subset \Frqf_n(S, X/U ),
\beta^{\lambda s_{\alpha}+(1-\lambda) x}\in  \Frqf_n(S\times\A^1, X/U ),\]
where $\beta = \alpha \cap \Frqf_n(S, X/U )$,
 $x=(x_i x_\infty^{d-1})$, and 
see  def. \ref{def:Phi^s} and def. \ref{def:Phi^lambdas} for the maps $(-)^{s}$ and $(-)^{\lambda s_{\alpha}+(1-\lambda) x}$.
\end{lemma}
\begin{proof}
The claim follows from proposition \ref{prop:hom-to-qf}.
\end{proof}

\begin{proof}[Proof of proposition \ref{prop:FrqfFrAEq}]

Denote the pointed sheaves $\Fr=\Fr_n(X/U)$ and $\Frqf=\Frqf_n(X/U)$

Consider the filtering set $A$ of finite sets $\alpha$ of sections $(\Phi_i)_{i\in \alpha}$,  of sheaf the $\Fr$.
So for each $\alpha \in A$ we canonically have a set of correspondences
\[
\Phi_i\in \Fr_n(S_i, X/U), S_i\in \Sm_k, i\in \alpha. \]
Note that $A$ is a set, since the category $\Sm_k$ is small.
Define  $\Fr_{\alpha}\subset \Fr$
as the smallest subpresheaves of $\Fr$ 
containing $\alpha$.
Define 
$\Frqf_{\alpha}=\Frqf\cap \Fr_\alpha$.
Then we have the pair of filtrations 
$$
\Fr=\varinjlim_{\alpha\in A} \Fr_{\alpha},\;
\Frqf=\varinjlim_{\alpha\in A} \Frqf_{\alpha}
$$

Lemma \ref{lm:agrreedPhi^s_maps} gives us the morphisms
\begin{gather*}\label{eq:Phi^s_maps}
(-)^{s_{\alpha}}\colon \Fr_{\alpha}\to \Frqf, \\
(-)^{\lambda s_{\alpha}+(1-\lambda) x}\colon \Fr_{\alpha}\times\A^1\to \Fr,\,
(-)^{\lambda s_{\alpha}+(1-\lambda) x}\colon \Frqf_{\alpha}\times\A^1\to \Frqf.
\end{gather*}

Then we consider the telescopes of the filtrations on $\Fr$ and $\Frqf$ (see the discussion above for the definition), 
and extend \eqref{eq:Phi^s_maps} to morphisms of pointed simplicial sheaves
\begin{gather*}
\mathrm{Tel}(\Fr_*) \to \Frqf_*,\\\label{eq:Phi^s_mapshomotopies}
\mathrm{Tel}(\Fr_*) \times \A^1 \to \Fr_*,
\mathrm{Tel}(\Frqf_*) \times \A^1 \to \Frqf_*.
\end{gather*}
To do this for an arbitrary linearly ordered subset $\{\alpha_0,\dots \alpha_r\}\subset A$ we define morphisms
\begin{equation*}\label{eq:Phi^s_mapsDelts}\begin{array}{llll}
\Fr_{\alpha}\times\Delta^{r} &\to& \Frqf ,&  \\
(c, (\lambda_0,\dots \lambda_{r})) & \mapsto & c^{s}, & s= \sum\limits_{i=0}^r \lambda_i s_{\alpha_i},
\end{array}\end{equation*}
and 
\begin{equation*}\label{eq:Phi^s_mapshomotopiesDelta}\begin{array}{llllll}
\Fr_{\alpha}\times\A^1 \times\Delta^{r}&\to& \Fr, &
\Frqf_{\alpha}\times\A^1 \times\Delta^{r}&\to& \Frqf, 
\\
(c, \lambda, (\lambda_0,\dots \lambda_{r})) & \mapsto & c^{\lambda s+(1-\lambda) x}, & 
(c, \lambda, (\lambda_0,\dots \lambda_{r})) & \mapsto & c^{\lambda s+(1-\lambda) x}.
\end{array}\end{equation*}

Then due to the homotopies \eqref{eq:Phi^s_mapshomotopies}
the compositions 
$$
\mathrm{Tel}(\Fr) \to \Frqf_*\to \Fr,\;
\mathrm{Tel}(\Frqf) \to \mathrm{Tel}(\Fr)\to \Frqf_*
$$
are $\A^1$-homotopy equivalent to the canonical maps 
\begin{equation}\label{eq:canmapTelFr}\mathrm{Tel}(\Fr)\to \Fr, \mathrm{Tel}(\Frqf) \to \Frqf\end{equation}
Thus since by lemma \ref{lm:OSetTel} the canonical morphisms \eqref{eq:canmapTelFr} are $\A^1$-equivalences,
the $\A^1$-equivalence
$\Frqf\simeq \Fr$
follows.
\end{proof}

Next we improve the Cone Theorem \ref{th:ConeTh} to the case of products of pairs.
\begin{corollary}\label{cor:ConeThPairProd}
Let $X$ and $Y$ be smooth affine $k$-schemes, and $U\subset X$ and $V\subset Y$ be open subschemes.

Then the canonical morphism 
\[\Fr( X/U \wedge (Y//V) ) \to \Fr( X/U\wedge (Y,V) )\]
is a motivic equivalence of pointed simplicial Nisnevich sheaves.
At the left side $(X,U)\times \mathcal Y$ denotes the simplicial object in $\Sm^\mathrm{pair}_\bullet(k)$
given by the open immersion of simplicial schemes
$U\times \mathcal Y\to X\times \mathcal Y$, and see def. \ref{def:FrcalX/calU}.
\end{corollary}
\begin{proof}
The claim follows form three equalities: 

1) By theorem \ref{th:ConeTh} we have the motivic equivalence of pointed simplicial sheaves
\[\Fr( X\times Y // (U\times Y\cup X\times V) ) \to \Fr( X/U\wedge Y/U )\]

2) By corollary \ref{cor:MayerVietoris} we have the Nisnevich local equivalence
$\Fr( U\times Y \amalg^\sim_{(U\times V)} X\times V )  \to \Fr( U\times Y\cup X\times V )$, see def. \ref{def:hocoeq} for $\amalg^\sim$.
Hence
\[\Fr( (X//U)\wedge (Y//U) ) \to \Fr( X\times Y // (U\times Y\cup X\times V) )\]
is a Nisnevich local equivalence.

3) By theorem \ref{th:ConeTh} again the morphism 
\[\Fr( X/U \wedge (Y//V) )\to \Fr( (X//U)\wedge (Y//V) )\]
is a motivic equivalence, since morphisms
\[\begin{array}{lll}
\Fr( X/U \wedge Y ) &\to& \Fr( (X//U)\times Y ),\\
\Fr( X/U \wedge V ) &\to& \Fr( (X//U)\times V )
\end{array}\]
are of such type.
%
%
\end{proof}

Now we apply the Cone Theorem (theorem \ref{th:ConeTh} and corollary \ref{cor:ConeThPairProd}) to get the computational results 
in the stable motivic homotopy category announced in the introduction. 
Assume that the base filed $k$ is perfect.

\begin{corollary}\label{cor:MSfreqiv}
Consider the canonical morphism of $S^1$-spectra of pointed simplicial Nisnevich sheaves
\begin{equation}\label{eq:Mfr(x//U)Mfr(X,U)}\Mfr(X//U)\to \Mfr(X,U),\end{equation}
see def. \ref{def:MfrMPMGfrAppensix} for $\Mfr$, and def. \ref{def:SimplCone} for $X//U$.
Then for any smooth affine $X$ and open $U$ the morphism 
\eqref{eq:Mfr(x//U)Mfr(X,U)}
is a levelwise Nisnevich local equivalence in positive degrees.
\end{corollary}
\begin{proof}
Theorem \ref{th:ConeTh} implies the motivic equivalence of pointed simplicial Nisnevich sheaves 
$\Fr( (X//U)\times K )\simeq \Fr( (X/U)\times K )$ for any simplicial set $K$. 
Hence we get the motivic equivalence 
$\MSfr(X//U)\to \MSfr(X,U)$, see def. \ref{}.
Then since $\Mfr=C^*(\MSfr)$, the morphism \eqref{eq:Mfr(x//U)Mfr(X,U)} is a motivic equivalence.
By proposition \ref{prop:motfibOmegaMfr(X,U)} 
it follows that Nisnevich local injective fibrant replacements $\Mfr(X//U)_f$ and $\Mfr(X,U)_f$
are levelwise motivically fibrant spectra in positive degrees.
So the claim follows.
\end{proof}

Now we assume in addition that the base field $k$ is perfect.
\begin{corollary}\label{cor:MPmotresolution}
The canonical morphism $\Sigma^\infty_{\mathbb P^1} (X/U)\to \MP(X,U)_f$ is the stable motivic weak equivalence of $\PP^{\wedge 1}$-spectra of pointed simplicial sheaves.
The $\PP^{\wedge 1}$-spectra $\MP(X,U)_f$ is a positively motivically fibrant $\Omega_{\mathbb P^1}$-spectrum.
See def. \ref{def:MfrMPMGfrAppensix} for $\MP$. 
\end{corollary}
\begin{proof}
By the first point of \cite[Theorem 10.1]{GP_MFrAlgVar}
we have the stable motivic weak equivalences
\begin{equation}\label{eq:SigmainfP=MP(X//U)}
\Sigma^\infty_{\PP^1}(X/U)\simeq \Sigma^\infty_{\PP^1}(X//U)\simeq \MP(X//U)_f
\end{equation}
Namely \cite[Theorem 10.1]{GP_MFrAlgVar} provides the second equivalence in the sequence above, and the first one is induced by the motivic equivalence $X//U\simeq X/U$.

By corollary \ref{cor:MSfreqiv} 
we have the levelwise Nisnevich local equivalence of $(\Gm,S^1)$-bi-spectra
\begin{equation}\label{eq:MGcS(X,U)=MGcS(X//U)}
C^*(\MGcSfr(X//U))\simeq C^*(\MGcSfr(X,U)) 
\end{equation}
see def. \ref{def:MGcSfr} for $\MGcSfr$.
Hence 
we have the levelwise Nisnevich local equivalence of $(\Gm^{\wedge 1}\wedge S^1)$-bi-spectra (see def. \ref{def:MGwSfr} for $\MGwSfr$)
\begin{equation*}
C^*(\MGwSfr(X//U))\simeq C^*(\MGwSfr(X,U))
,\end{equation*}
since the above morphism of $(\Gm^{\wedge 1}\wedge S^1)$-spectra is the diagonal of bi-spectra morphism \eqref{eq:MGcS(X,U)=MGcS(X//U)}.
Then due to the $\A^1$-local equivalences of simplicial schemes
$(\A^1//\Gm)\simeq (\mathrm{pt}//\Gm)\simeq \Gm^{\wedge 1}\wedge S^1$ 
by remark \ref{rem:MGwSfr=McTfr}
we have the levelwise Nisnevich local equivalence
\begin{equation*}
C^*(\McTfr(X//U))\simeq C^*(\McTfr(X,U))
,\end{equation*}
see def. \ref{def:McTfr} for $\McTfr$.
Now by corollary \ref{cor:ConeThPairProd} and remark \ref{rem:MTtoMPeq}
we get the levelwise Nisnevich local equivalences of $T$-spectra and $\PP^{\wedge 1}$-spectra of pointed simplicial sheaves
\begin{multline}\label{eq:MT(X,U)=MT(X//U),MP(X,U)=MP(X//U)}
C^*(\MTfr(X//U))\simeq C^*(\MTfr(X,U)),\\ C^*(\MPfr(X//U))\simeq C^*(\MPfr(X,U)),
\end{multline}
see def. \ref{def:MTfr} and def. \ref{def:MPfr} or def. \ref{def:MPfrMP}.
Hence by def. \ref{def:MPfrMP} we get the levelwise schemewise simplicial equivalence of $\PP^{\wedge 1}$-spectra of pointed simplicial sheaves,
\begin{equation}\label{eq:MP(X//U)=MP(X,U)}\MP(X//U)_f\simeq \MP(X,U)_f\end{equation} and combining with equivalence \eqref{eq:SigmainfP=MP(X//U)} 
we get the first clam of the corollary.

The second point of \cite[Theorem 10.1]{GP_MFrAlgVar} implies that $\MP(X//U)_f$ is positively motivically fibrant $\Omega$-spectrum, hence by \eqref{eq:MP(X//U)=MP(X,U)} the spectrum $\MP(X,U)_f$ is positivley motivically fibrant  as well.
\end{proof}
\begin{corollary}\label{cor:MGfrresolution}
The canonical morphism of bi-spectra \[\Sigma^\infty_{\Gm}\Sigma^\infty_{S^1}(X/U)\to \MGfr(X,U)_f\] is a stable motivic weak equivalence, and the spectrum $\MGfr(X,U)_f$ is motivically fibrant. See def. \ref{def:MfrMPMGfrAppensix} for $\MGfr(X,U)$.
\end{corollary}
\begin{proof}
By the definition 
$\MGfr(-)=C^*(\MGcSfr(-))$.
So by \eqref{eq:MGcS(X,U)=MGcS(X//U)}
\[\MGfr(X//U)_f\to \MGfr(X,U)_f\]
is a levelwise schemewise simplicial equivalence of simplicial sheaves.
So it is enough to prove the claim for the spectrum $\MGfr(X//U)_f$. 

By 
\cite[Theorem 11.4]{GP_MFrAlgVar} 
the canonical morphism 
\[\Sigma^\infty_{\Gm}\Sigma^\infty_{S^1}(X/U)\to \MGfr(X//U)_f\]
is stable motivic weak equivalence.

it follows by \cite[Corollary 7.5]{GP_MFrAlgVar}
that 
each term $C^*( \Fr( (X//U)\wedge \Gm^{\wedge i}\wedge S^j ) )_f$, $j>0$, of the spectrum $\MGfr(X//U)_f$ is motivically fibrant.

The bi-spectra $\MGfr(X)_f$ and $\MGfr(U)_f$ are $\Omega_{S^1}$-bi-spectra by \cite[Theorem 6.5]{GP_MFrAlgVar} and are $\Omega_{\Gm}$-bi-spectra by \cite[Theorem A]{AGP-FrCanc}.
Then it follows that $\MGfr(X//U)_f$
is an $\Omega_{(\Gm,S^1)}$-bi-spectrum.
\end{proof}
\begin{corollary}\label{cor:MfrcomputinfloopSH^S1}
The $S^1$-spectrum of pointed simplicial sheaves
$\Mfr(X,U)_f$ is a motivically fibrant $\Omega$-spectrum in positive degrees
and has the homotopy type of 
$\Omega^\infty_{\mathbb G_m}\Sigma^\infty_{\mathbb G_m}\Sigma^\infty_{S^1}(X/U)$ in $\mathbf{SH}^{S^1}(k)$. 
\end{corollary}
\begin{proof}
The claim follows form corollary \ref{cor:MGfrresolution}, since
$\Mfr(X,U)_f$ is the zero $\Gm$-row of the spectrum $\MGfr(X,U)_f$.
\end{proof}

\section{Appendix. Framed motivic spectra.}
In the appendix we summarise the definitions, constructions, and lemmas on framed correspondences and framed motivic spectra
used in the text. Some of definitions recovers the ones from the section \ref{sect:FrCorandMot}.





\begin{definition}
(i)
%
We call by the category of \emph{pointed smooth open pairs} $\Sm^\mathrm{pair}_\bullet(k)$ the category with objects $(X,Z,U)$ given by $X\in \Sm_k$, closed subscheme 
$Z\subset X$, and open 
$U\subset X$;
a morphism 
form $(X_1,Z_1,U_1)$ to $(X_2,Z_2,U_2)$ is given by $f\colon X_1\to X_2$, $f^{-1}(Z_2)\supset Z_1$, $f^{-1}(U_2)\supset U_1$.

Define the subcategory of \emph{pointed schemes} $\Sm_\bullet(k)$ in $\Sm^\mathrm{pair}_\bullet(k)$ spanned
by objects of the type $(X,Z,\emptyset)$;
and the subcategory of \emph{open pairs} $\Sm^\mathrm{pair}_k$ spanned by objects of the type $(X,\emptyset,U)$.

(ii)
Define the smash-product 
$-\wedge -\colon \Sm^\mathrm{pair}_\bullet(k)\times \Sm^\mathrm{pair}_\bullet(k) \to  \Sm^\mathrm{pair}_\bullet(k)\colon$
\[ ( (X_1,Z_1,U_1) , (X_1,Z_1,U_1) ) \mapsto  (X_1\times X_2 , Z_3,U_3 ), \]
$Z_3=(Z_1\times X_2)\cup (Z_2\times X_1)$, $U_3=(U_1\times X_2)\cup (U_2\times X_1) $.
\end{definition}

\begin{definition}\label{def:Fr(X/U)pointed}
Let $(X,Z,U)\in \Sm^\mathrm{pair}_\bullet(k)$.
Define the pointed sheaf 
\[\Fr( (X,Z)/U )=\colim( *\leftarrow \Fr( Z/(U\times_X Z) )\to \Fr(X/U) )\in \Shv_\bullet,\] 
see def. \ref{def:Fr(X/U)} for $\Fr(X/U)$.
%
\end{definition}

\begin{definition}\label{def:FrcalX/calU}
Denote by $\Delta^\mathrm{op}\Sm^\mathrm{pair}_\bullet(k)$ 
the category of simplicial objects $\Sm^\mathrm{pair}_\bullet(k)$.
Define the functor 
\[\Fr\colon \Delta^\mathrm{op}\Sm^\mathrm{pair}_\bullet(k)\to \sShv_\bullet \]
as the functor induced by 
$
\Sm^\mathrm{pair}_\bullet(k)\to \Shv_\bullet \colon 
(X,Z,U)  \mapsto  \Fr( (X,Z)/U )
.$ 
%
\end{definition}

\begin{definition}
A set of morphisms $X\times \A^n\to X$ for all schemes $X$
we call by \emph{strong naive $\A^1$-local equivalences of schemes}.

A morphism of simplicial scheme $f\colon \mathcal X\to \mathcal Y$ is called as
\emph{strong naive $\A^1$-local equivalence} iff
$f$ is termwise \emph{strong naive $\A^1$-local equivalence of schemes}.
\end{definition}

\begin{remark}\label{rem:Fr(A1eqiv)}
Any strong naive $\A^1$-local equivalence of simplicial smooth schemes $\mathcal X\to \mathcal Y$ 
induces the $\A^1$-local equivalence 
$\Fr(\mathcal X)\to \Fr(\mathcal Y) $. 

More generally a strong naive $\A^1$-local equivalence of simplicial smooth pairs $(\mathcal X_1,\mathcal U_1)\to (\mathcal X_2,\mathcal U_2)$ 
induces the 
the schemewise simplicial equivalence 
\[C^*(\Fr(\mathcal X_1/\mathcal U_1))\to C^*(\Fr(\mathcal X_2/\mathcal U_2)).\]

\end{remark}

\begin{definition}\label{def:MSfr}
Let $(X,U)$ be a pair given by smooth affine scheme $X$ and open subscheme $U\subset X$, or a simplicial pair.
Define an $S^1$-spectrum of pointed simplicial sheaves 
\[\MSfr(X,U) = ( \Fr(X/U), \Fr( (X/U) \wedge S^1), \dots \Fr( (X/U) \wedge S^i), \dots )\]
%
\end{definition}

\begin{definition}\label{def:MGcSfr}
Let $(X,U)$ be as in def. \ref{def:MSfr}.
Define an $(S^1,\Gm^{\wedge 1})$-spectrum of pointed simplicial sheaves 
$\MGcSfr(X,U)$ as the bi-spectrum with the terms \[\Fr( (X/U) \wedge S^i\wedge \Gm^{\wedge j})\],
where $\Gm^{\wedge 1}=\Gm//\{1\}$, see def. \ref{def:SimplCone}, and 
see def. \ref{def:FrcalX/calU} for $\Fr(\Gm^{\wedge 1})$. 
%
Define $\MGcSfr(X)=\MGcSfr(X,\emptyset)$.
\end{definition}

\begin{definition}\label{def:MGwSfr}
Let $(X,U)$ be as in def. \ref{def:MSfr}.
Define an $(\Gm^{\wedge 1}\wedge S^1)$-spectrum of pointed simplicial sheaves 
$\MGwSfr(X,U)$ as
\[( \Fr(X/U), \Fr( (X/U) \wedge \Gm^{\wedge 1}\wedge S^1), \dots \Fr( (X/U) \wedge \Gm^{\wedge i}\wedge S^i), \dots ).\]
%
\end{definition}
\begin{remark}
$\MGwSfr(X,U)$ is the diagonal of the bi-spectrum $\MGcSfr(X,U)$.
\end{remark}

\begin{definition}\label{def:McTfr}
Let $(X,U)$ be as in def. \ref{def:MSfr}.
Define an $(\A^1//(\A^1-0))$-spectrum of pointed simplicial sheaves $\McTfr(X,U)$ as
\[( \Fr(X/U), \Fr( (X/U) \wedge (\A^1//(\A^1-0)) ), \dots  \\ \Fr( (X/U) \wedge (\A^1//(\A^1-0))^{\wedge i} ), \dots ).\]
\end{definition}

\begin{remark}\label{rem:MGwSfr=McTfr}
It follows by remark \ref{rem:Fr(A1eqiv)}
that 
the terms of spectra $\MGwSfr(X,U)$ and $\McTfr(X,U)$ are levelwise $\A^1$-local equivalent.
Actually, 
we have a strong naive $\A^1$-local equivalence of pointed simplicial schemes
$\A^1//\Gm\to \mathrm{pt}//\Gm$,
and an simplicial equivalence of pointed simplicial schemes
$(\mathrm{pt}//\Gm)\simeq {\Gm^{\wedge 1}}\wedge S^1.$
Hence
there is an $\A^1$-local equivalence of pointed simplicial presheaves
$\Fr( (X/U)\wedge (\A^1//\Gm)^{\wedge i} )\to \Fr( (X/U)\wedge (\Gm^{\wedge i}\wedge S^i) )$.
\end{remark}

\begin{definition}\label{def:MTfr}
Let $(X,U)$ be as in def. \ref{def:MSfr}.
Define an $(\A^1/(\A^1-0))$-spectrum of pointed simplicial sheaves 
\[\MTfr(X,U) = ( \Fr(X/U), \Fr( (X/U) \wedge T ), \dots \Fr( (X/U) \wedge T^i ), \dots ),\]
where $T$ denotes the open pair $(\A^1,\A^1-0)$, see \eqref{eq:(X1,U1)wedge(X2,U2)} for the smash-product.
the structure maps are given by
\[\begin{array}{lll}
\Fr( S,  (X/U) \wedge T^i ) &\to & 
Hom_\bullet(\A^1/(\A^1-0)\times S, \Fr( (X/U) \wedge T^{i+1} ) ) \\
(Z,V,\varphi_1,\dots \varphi_n, g) & \mapsto & 
(0\times Z,\A^1\times V,\varphi_1,\dots \varphi_n, (t,g) ),
\end{array}\]
where 
$Hom_\bullet$ denotes the pointed hom-set in 
$\sShv_\bullet$,
\begin{equation}\label{eq:SigmainOmega(Fr)}
(0\times Z,\A^1\times V,\varphi_1,\dots \varphi_n, (t,g) )\in \Fr(\A^1\times S, (X/U) \wedge T^{i+1} ),
\end{equation}
where the function $t$ is the coordinate on $\A^1$, 
and \eqref{eq:SigmainOmega(Fr)} is considered as morphism of simplicial sheaves $\A^1\times S\to \Fr( (X/U) \wedge T^{i+1} )$, that passes throw $\A^1/(\A^1-0)\times S$.

\end{definition}

The canonical morphism of pointed sheaves 
$\PP^{\wedge 1}\to \A^1/(\A^1-0)$ induces the functor  \[\nu\colon \Spec_{\A^1/(\A^1-0)}(\sShv_\bullet)\to \Spec_{\PP^{\wedge 1}}(\sShv_\bullet)\]
form the category of $(\A^1/(\A^1-0))$-spectra to $\PP^{\wedge 1}$-spectra of pointed simplicial sheaves.

\begin{definition}\label{def:MPfr}
Let $(X,U)$ be as in def. \ref{def:MSfr}.
Define $\PP^{\wedge 1}$-spectra \[\MPfr(X,U)=\nu(\MTfr(X,U)), \MPfr(X)=\nu(\MTfr(X)).\]
\end{definition}

\begin{remark}\label{rem:MTtoMPeq}
It follows immediate form the definition that any (stable motivic) equivalence of $(\A^1/(\A^1-0))$-spectra
$\MTfr(\mathcal X,\mathcal U)\simeq \MTfr(\mathcal X^\prime,\mathcal U^\prime)$, for some simplicial pairs of smooth schemes $(\mathcal X,\mathcal U)$, $(\mathcal X^\prime, \mathcal U^\prime)$, induces the (stable motivic) equivalence of $\PP^{\wedge 1}$-spectra $\MPfr(\mathcal X,\mathcal U)\simeq \MPfr(\mathcal X^\prime,\mathcal U^\prime)$.
\end{remark}

\begin{definition}
For any $\mathrm{M}^{*}_\fr$ from def. \ref{def:MSfr}-\ref{def:MPfr}
set $\mathrm{M}^{*}_\fr(X)=\mathrm{M}^{*}_\fr(X,\emptyset)$.
\end{definition}

\begin{definition}\label{def:MfrMPMGfrAppensix}
Define 
the following 
$S^1$-spectrum, $(\Gm,S^1)$-bi-spectrum and $\PP^{\wedge 1}$-spectrum 
of pointed simplicial sheaves
\begin{gather*}\Mfr(-)=C^*(\MSfr(-)), \MGfr(-) = C^*(\MGcSfr(-)), \\ \MP(X,U)=C^*(\MPfr(-))\end{gather*}
respectively.
The notation 
are agreed with \cite{GP_MFrAlgVar}.
\end{definition}


\begin{proposition}\label{prop:motfibOmegaMfr(X,U)}
For any smooth scheme $X$ over an infinite perfect field $k$ and open subscheme $U$ the spectrum
$\Mfr(X,U)_f$ is a levelwise motivically fibrant $\Omega_{S^1}$-spectrum in positive degrees.
\end{proposition}
\begin{proof}
By \cite[corollary 7.5]{GP_MFrAlgVar} the spectrum $\Mfr( (X,U)\wedge S^1 )_f$ is levelwise motivically fibrant $\Omega_{S^1}$-spectrum.
Hence the claim follows, since by construction $\Mfr( (X,U)\wedge S^1 )_f$ is equal to the shift of $\Mfr( X,U )_f$.
\end{proof}

\bibliographystyle{unsrt} 
\bibliography{ifFrp-JPAA-ref}

\end{document}